\documentclass{article}%
\usepackage{graphicx}
\usepackage{amsmath,amssymb,amsfonts}%
\usepackage{theorem}%
\usepackage{color}%
\usepackage{hyperref}
\usepackage[font={small, it}]{caption}
\usepackage{caption}
\usepackage{subcaption}
\usepackage{float}
\usepackage{tikz-cd}
\tikzcdset{scale cd/.style={every label/.append style={scale=#1},
    cells={nodes={scale=#1}}}}
\usepackage{forest}
\usepackage{tikz-qtree}
\usetikzlibrary{trees}    
\theoremstyle{change}%
\usepackage{cite}

\newtheorem{definition}{Definition:}[section]%
\newtheorem{proposition}[definition]{Proposition:}%
\newtheorem{theorem}[definition]{Theorem:}%
\newtheorem{lemma}[definition]{Lemma:}%
\newtheorem{corollary}[definition]{Corollary:}%
{\theorembodyfont{\rmfamily}\newtheorem{remark}[definition]{Remark:}}%
{\theorembodyfont{\rmfamily}}%

\newenvironment{proof}
{{\bf Proof:}}
{\qquad \hspace*{\fill} $\Box$}%

\newcommand{\tr}{\operatorname{tr}}%

\newcommand{\inner}{\operatorname{int}}%
\newcommand{\rme}{\mathrm{e}}%

\newcommand{\CC}{\mathcal{C}}%
\newcommand{\OC}{\mathcal{O}}%
\newcommand{\UC}{\mathcal{U}}%
\newcommand{\PC}{\mathcal{P}}

\newcommand{\N}{\mathbb{N}}%
\newcommand{\R}{\mathbb{R}}%

\begin{document}

\title{Control sets of linear control systems on $\R^2$. The real case.}
\author{V\'{\i}ctor Ayala  \\Universidad de Tarapac\'a\\Instituto de Alta Investigaci\'on\\Casilla 7D, Arica, Chile
\and Adriano Da Silva\thanks{Supported by Proyecto UTA Mayor Nº 4768-23} \\Departamento de Matem\'atica,\\Universidad de Tarapac\'a - Iquique, Chile\and Anderson F. P. Rojas \\
 		Instituto de Matem\'{a}tica\\
		Universidade Estadual de Campinas, Brazil.}
\date{\today }
\maketitle

\begin{abstract}
In this paper, we study the dynamical behavior of a linear control system on $\R^2$ when the associated matrix has real eigenvalues. Different from the complex case, we show that the position of the control zero relative to the control range can have a strong interference in such dynamics if the matrix is not invertible. In the invertible case, we explicitly construct the unique control set with a nonempty interior.
\end{abstract}

\maketitle

	

\section{Introduction}

Starting in the 1960s and motivated by the Cold War, control theory is nowadays a mathematical discipline with a high degree of maturity and enormous influence in the real world. In particular, for the class of linear control systems on finite-dimensional vector spaces, the number of applications is vast, acting in almost all areas of knowledge. A control system allows one to influence the behavior of an original system to achieve a specific goal. Of course, this model represents a solid understanding of how the system responds to different commands. Assume that a control system $\Sigma$ of the object to control is available, for example, from physical principles and engineering specifications. As a mathematical model, $\Sigma$ represents a way to influence its own states in order to reach a prescribed target using different strategies, i.e., the controls. We also need to consider constraints in control and target. For example, a government wants a country with full employment and a zero inflation rate. However, in the real world, a realistic target could be unemployment of less than 8\% and an inflation rate of less than 5\%, \cite{Macki}. Alternatively, in biomedicine, the number of therapeutic agents given to treat cancer has practical constraints. The drugs to eliminate cancer cells from a biological host should be given with consideration of their collateral damage effects, \cite{Shattler}. The system's evolution represents the crucial mathematical tool to face these problems.
Generally, a control system $\Sigma$ is given by a state space $M$, a differential manifold, and a family of ordinary differential equations (vector fields) determined by all possible controls. Assume the system $\Sigma$ achieves the desired target state $y\in M$, from a prescribed initial condition $x \in M$. Any process can be subject to small perturbations that move the system away from the target. Therefore, the reachable set of $\Sigma$ from the point $x$, i.e., the set that contains the elements reachable from $x$ through all available flows and any positive time, should contain a neighborhood of $y$. Even better, to be equal to the entire space of state $M$. In the last case, we say that the system $\Sigma$ is controllable from $x$. However, the controllability property of a control system is very rare. More realistic is the notion of a control set. Roughly speaking, control sets are the maximal regions of $M$ where controllability holds in its interior. The mathematical model of a control system allows us to understand the behavior of the system's dynamics through conceptual analysis.
The class of the linear control system is well-known and has been studied by several authors in different contexts (see, for instance, \cite{AMS, Jurdjevic, Kalman, Sontag, Wonham}). They have a straightforward structure and can appear directly from the dynamics of a specific problem, or as a linearization of a nonlinear model. In the last case, the analysis works locally for the original system. By definition, a linear control system on the $n$-dimensional real vector space $\R^n$, is given by
\begin{flalign*}
&&\dot{v}(t) = Av(t)+Bu(t), \hspace{.5cm}v(t)\in\R^n,\hspace{.5cm} u(t)\in\Omega, \hspace{.5cm}t\in\R, &&\hspace{-1cm}\left(\Sigma_{\R^n}\right)
\end{flalign*}
where $\Omega\subset \R^m$ is a bounded subset, $A \in\R^{n\times n}$ and $B\in \R^{n\times m}$ are matrices. If $0\in\inner\Omega$ and $\Sigma_{\R^n}$ satisfy the Lie Algebra Rank Condition (LARC), which means
$$\mathrm{rank}[A|B]=\mathrm{rank}\left[A^{n-1}B| A^{n-2}B|\cdots |B\right]=n,$$
it is well known that $\Sigma_{\R^n}$ admits a unique control set $\CC_{\R^n}$ with a nonempty interior. Such a control set is bounded if and only if the matrix $A$ has only eigenvalues with nonzero real parts, and it is closed (open) if and only if $A$ has only eigenvalues with nonnegative (nonpositive) real parts (see, for instance, \cite[Example 3.2.16]{FCWK}).

In this paper, we consider a linear control system $\Sigma_{\R^2}$ on the plane satisfying the LARC and whose associated matrix $A$ has a pair of real eigenvalues. Our aim is twofold. Firstly, under the assumption that $0\in\inner\Omega$, explicitly characterize the unique control set $\CC_{\R^2}$ in a computable way. Secondly, analyze what happens under the assumption that $0\notin\inner\Omega$. Recently in \cite{VAAS}, the authors explicitly computed the unique control set with a nonempty interior for the case where the associated matrix $A$ has a pair of complex eigenvalues. It turns out that the closure of the control set coincides with the region delimited by a computable periodic orbit $\OC$ of the system. Moreover, the assumption on $0\in\inner\inner\Omega$ has no meaning. In the real case, however, such a condition can have a strong influence if $\det A=0$, as shown in Theorems 3.1 and 3.3.

It is worth pointing out that the results presented here are not limited to dimension two. In fact, for a general LCS on $\R^n$, if the matrix $A$ is semisimple, i.e., its complexification is diagonalizable, the system can be decomposed into blocks of dimensions $1$ and $2$ determined by the eigenvalues of $A$. In particular, it induces a linear control system on any A-invariant plane. Its restriction to such a plane is one of the systems considered in the present paper or in \cite{VAAS}. Let us also mention that this theory has been generalized to the class of linear control systems on Lie groups (see \cite{AMS}).

The article is organized as follows: after the preliminaries, Section 2 introduces the definition of a linear control system on the plane, the positive and negative orbits, and the notion of control sets. The last section, Section 3, studied the system's control sets and the controllability property. Here, we introduce our central hypothesis: the eigenvalues of $A$ are real. Recall that the complex cases were treated in \cite{VAAS}. Under this assumption, the analysis will be divided by considering the possibilities for the determinant and trace of A. It turns out that the different cases come from the relative position of the control ${\bf u}\equiv 0$ concerning the set $\Omega$: inside of $\Omega$ as an interior point, in the boundary of $\Omega$, and outside of $\Omega$. There are four cases to consider: $\det A=0$ and $\tr A=0$; $\det A=0$ and $\tr A$ different from zero. The other two cases come from $\det A$ being different from zero with $\tr A<0$ and $\tr A>0$. All the existing control sets are explicitly described.

\section{Linear control systems on $\R^2$ and control sets}

A linear control system (LCS) on $\R^2$ is determined by the family of ODEs
	\begin{flalign*}
		&&\dot{v}(t)=Av(t)+u(t)\eta, \;\;\;\;u(t)\in\Omega, \hspace{.5cm}t\in\R,&&\hspace{-1cm}\left(\Sigma_{\R^2}\right)
	\end{flalign*}
where $A\in\mathfrak{gl}(2)$, the {\it control range} $\Omega:=[u^-, u^+]$ with $u^-<u^+$, and $\eta\in\R^2$ is a nonzero vector. 

The family of the {\it control functions} $\UC$ is the set of all {\it piecewise constant functions} whose image is contained in $\Omega$. The {\it solution} of the system $\Sigma_{\R^2}$ starting at the point $v\in\R^2$, associated with the control $ {\bf u}\in\UC$ is the unique piecewise differentiable curve $s\in\R\mapsto\varphi(s, v, {\bf u})$ satisfying 
$$\frac{d}{ds}\varphi(s, v, {\bf u})=A\varphi(s, v, {\bf u})+{\bf u}(s)\eta.$$
By the very definition, any solution of $\Sigma_{\R^2}$ is obtained by concatenation of constant control functions. In particular, when $\det A\neq 0$ the solutions for constant controls $u\in\Omega$ are given by
\begin{equation}
    \label{solutionneqzero}
    \varphi(t, v, u)=\rme^{t A}(v-v(u))+v(u), \hspace{.5cm}\mbox{ where } \hspace{.5cm}v(u):=-uA^{-1}\eta,
\end{equation}
are the equilibria of the system.

For $v\in\R^2$, we define the  {\it positive} and {\it negative orbits} of $\Sigma_{\R^2}$, respectively, as the sets 
$$\OC^+(v):=\{\varphi(t, v, {\bf u}), \;t\geq 0, {\bf u}\in\UC\}\hspace{1cm}\mbox{ and }\hspace{1cm}\OC^-(v):=\{\varphi(t, v, {\bf u}), \;t\leq 0, {\bf u}\in\UC\}.$$

We will say that $\Sigma_{\R^2}$ satisfies the LARC if $\langle A\zeta, \theta\zeta\rangle\neq 0$, where $\theta$ stands for the counter-clockwise rotation of $\frac{\pi}{2}$-degrees. Equivalently, $\Sigma_{\R^2}$ satisfies the LARC if and only if $\zeta$ is not an eigenvector of $A$. Such an assumption assures, in particular, that the positive and negative orbits have a nonempty interior.

\begin{definition}
\label{control}
A nonempty subset $\CC\subset\R^2$ is called a {\it control set} of $\Sigma_{\R^2}$ if it satisfies:
\begin{enumerate}
    \item For any $v\in \CC$ there exists ${\bf u}\in\UC$ such that $\varphi(\R^+, v, {\bf u})\subset \CC$;
    \item For any $v\in \CC$ it holds that $\CC\subset\overline{\OC^+(v)}$;
    \item $\CC$ is maximal w.r.t. set inclusion satisfying 1. and 2. 
\end{enumerate}

\end{definition}

We say that $\Sigma_{\R^2}$ is {\it controllable} if $\CC=\R^2$ is the only control set of $\Sigma_{\R^2}$.

The following result states that subsets with a nonempty interior that are maximal with the second condition in \ref{control} are control sets. We will use this fact to prove our main results. Its proof can be found at \cite[Proposition 3.2.4]{FCWK}.

\begin{proposition}
    \label{int}
    Let $D\subset\R^2$ be a set that is maximal with the property that for all $v\in D$ one has $D\subset\overline{\OC^+(v)}$ and suppose that the interior of $D$ is nonvoid. Then $D$ is a control set.
\end{proposition}

In the next section, we analyze the dynamics of the solutions of $\Sigma_{\R^2}$ to obtain a characterization of its control sets. In particular, the assumption that $0\in\inner\Omega$ plays an essential role in the shape of the control sets.

\section{Controllability and the control sets of $\Sigma_{\R^2}$.}

In this section, we study the control sets and the controllability property of $\Sigma_{\R^2}$. We will assume here that the eigenvalues of $A$ are real, since the complex case was treated in \cite{VAAS}. Under this assumption, the analysis will be divided by considering the possibilities for the determinant and trace of $A$.

\subsection{The case $\det A=0$.}

In this section we treat the case where $\det A=0$. In order to do that, we analyze the two subcases related with the trace of $A$.

\subsubsection{The subcase $\tr A=0$}

In this case, the LARC implies that $\{A\zeta, \zeta\}$ is a basis, and on such a basis, the matrix $A$ is written as
$A=\left(\begin{array}{cc}
    0 &  1\\
    0 & 0
\end{array}\right).$

Moreover, the solutions of $\Sigma_{\R^2}$ for constant controls, are given as
$$\varphi(t, v_0, u)=v_0+t(Av_0+u\zeta)+u\frac{t^2}{2}A\zeta,$$
which, on the previous basis, have the form
$$\varphi(t, v_0, u)=\left(x_0+ty_0+u\frac{t^2}{2}, y_0+ut\right), \hspace{.5cm}\mbox{ where }\hspace{.5cm}v_0=(x_0, y_0).$$
For any $u\in\Omega$, let us define the function 
    $$F_u:\R^2\rightarrow\R, \hspace{1cm}F_u(x, y):=y^2-2ux.$$
    It holds that 
    $$F_u(\varphi(t, v_0, u))=(y_0+ut)^2-2u\left(x_0+ty_0+u\frac{t^2}{2}\right)=y_0^2+2y_0ut+u^2t^2-(2ux_0+2uy_0t+u^2t^2)$$
    $$=y_0^2-2ux_0=F_u(v_0).$$
    It is not hard to see that the image of the curve 
    $t\in\R\mapsto \varphi(t, v_0, u)$ coincides with $F_u^{-1}(F_u(v_0))$.

\begin{theorem}
\label{polynomial}
    If the LCS $\Sigma_{\R^2}$ satisfies the LARC and $\det A=\tr A=0$, it holds:
    \begin{itemize}
        \item[(a)] $0\in\inner\Omega$ and $\Sigma_{\R^2}$ is controllable;
        \item[(b)] $0\in\partial\Omega$ and $\R\cdot A\zeta$ is a continuum of one-point control sets.
    \item[(c)] $0\notin\Omega$ and $\Sigma_{\R^2}$ does not admit any control set.
    \end{itemize}
\end{theorem}

\begin{proof} (a) For $u\in\Omega$ and $v\in\R^2$, the interior of the region is delimited by the parabola $F_{u}^{-1}(F_u(v))$ is the set
   $$\PC_{u, v}^+:=\{w\in\R^2; F_{u}(w)< F_{u}(v)\}.$$
    Let $u_0<0<u_1$ and consider $v_0, v_1\in\R^2$ arbitrary points. A trajectory connecting $v_0$ to $v_1$ is constructed as follows:
    
    \begin{enumerate}
        \item From $v_0$ to $v_0'\in \PC_{u_1, v_1}^+$: 
        
        By considering $t_0>0$ large enough, it holds that
        $$\varphi_2(t_0, v_0, u_0)=y_0+u_1t_0>0, \hspace{.5cm} \mbox{ since }\hspace{.5cm}u_1>0.$$
        Now, with control constant equal to zero we get that
        $$F_{u_1}(\varphi(t, \varphi(t_0, v_0, u_1), 0)=[\varphi_2(t_0, v_0, u_1)]^2-2u_1[\varphi_1(t_0, v_0, u_1)+t\varphi_2(t_0, v_0, u_1)]$$
        $$=F_{u_1}(\varphi(t_0, v_0, u_0))-2u_1 t\varphi_2(t_0, v_0, u_0)\rightarrow-\infty, \hspace{.5cm}\mbox{ as }\hspace{.5cm}t\rightarrow+\infty.$$
        Consequently, there exists $t_1>0$ such that 
        $$F_{u_1}(\varphi(t_1, \varphi(t_0, v_0, u_1), 0))< F_{u_1}(v_1)\hspace{.5cm}\iff\hspace{.5cm}v_0':=\varphi(t_1, \varphi(t_0, v_0, u_1), 0)\in\PC^+_{u_1, v_1}.$$

\item From $v'_0\in \PC_{u_1, v_1}^+$ to $v_1$:    Since  
		$$t\rightarrow+\infty \hspace{.5cm}\implies\hspace{.5cm}\varphi_1(t, v'_0, u_0)\rightarrow-\infty\hspace{.5cm}\mbox{ and }\hspace{.5cm}\varphi_2(t, v'_0, \omega_0)\rightarrow-\infty,$$
	 there exists $t_2>0$ such that $v_0'':=\varphi(t_2, v'_0, u_0)$ belongs to the boundary $\partial \PC_{u_1, v_1}^+= F_{u_1}^{-1}(F_{u_1}(v_1))$ and we have to consider the following possibilities:

  \begin{itemize}
      \item[(i)]  $v_1\in \PC_{u_0, v_0'}^+$: In this case, we necessarily have $\varphi_2(t_2, v'_0, u_0)\leq y_1$. Moreover, the fact that the solution $t\mapsto \varphi(t, v''_0, u_1)$ lies on the parabola $F_{u_1}^{-1}(F_{u_1}(v_1))$ and that
	    $$t\rightarrow+\infty\hspace{.5cm}\implies\hspace{.5cm} y_0''+u_1t=\varphi_2(t, v_0'', u_1)\rightarrow+\infty,$$
	  imply the existence of $t_3>0$ such that $\varphi_2(t_3, v''_0, u_1)=y_1$, and hence, $\varphi(t_3, v''_0, u_1)=v_1$. By concatenation, we obtain a trajectory from $v_0$ to $v_1$.
		
		\item[(ii)]  $v_1\notin \PC_{u_0, v_0'}^+$: In this case, $F_{u_0}(v_1)\geq F_{u_0}(v'_0)$ and $$t\rightarrow+\infty\hspace{.5cm}\implies\hspace{.5cm}\varphi_1(t, v_0'', u_1)\rightarrow+\infty\hspace{.5cm}\mbox{ and }\hspace{.5cm}\varphi_2(t, v_0'', u_1)\rightarrow+\infty$$
  imply the existence of $t_3>0$ such that 
  $$\varphi(t_3, v_0'', u_1)\in F_{u_0}^{-1}(F_{u_0}(v_1))\cap F_{u_1}^{-1}(F_{u_1}(v_1)).$$
  Since $v_1$ also belongs to this intersection, there exists $t_4\geq 0$ such that 
  $$\varphi(t_4, \varphi(t_3, v_0'', u_1), u_0)=v_1,$$
  and again, by concatenation, we get the desired trajectory, showing the assertion.
   \end{itemize}
  \end{enumerate}

  (b) We start by noticing that, if $0\in\Omega$, the singleton $\{\alpha A\zeta\}$ satisfies conditions (a) and (b) in Definition \ref{control}  for any $\alpha\in\R$. Therefore, there exists a control set $\CC_{\alpha}$ with $\alpha A\zeta\in\CC_{\alpha}$. Now, for any ${\bf u}\in\UC$ and $t>0$, there exists $t_1, \ldots, t_n>0$ such that 
  $${\bf u}(s)=u_i\in\Omega\hspace{.5cm}\mbox{ for }\hspace{.5cm} s\in \left[\sum_{j=0}^{i-1}t_j, \sum_{j=0}^{i}t_j\right)\hspace{.5cm}\mbox{ with  }\hspace{.5cm}t_0=0\hspace{.5cm}\mbox{ and }\hspace{.5cm} t=\sum_{j=0}^{n}t_j.$$
  Therefore, if $v=(x, y)$, then
  \begin{equation}
      \label{solution2}
  \varphi_2(t, v, {\bf u})=y+\sum_{j=1}^nu_it_i.
    \end{equation}
    Consequently, if $0\in\partial\Omega$,  equation (\ref{solution2}) implies that
  $$\varphi(t^0_k, v_0, {\bf u}^0_k)\rightarrow v_1, \hspace{.5cm}\mbox{ and }\hspace{.5cm}\varphi(t^1_k, v_1, {\bf u}^1_k)\rightarrow v_0, \hspace{.5cm}\mbox{ with }\hspace{.5cm}t_k^0, t_k^1\geq 0\hspace{.5cm}\implies\hspace{.5cm}v_0=v_1\in \R\times\{0\},$$
  and hence $\CC_{\alpha}=\{\alpha A^{-1}\zeta\}$ are the control sets of $\Sigma_{\R^2}$.

  \bigskip

  (c) In the same way, by using equation (\ref{solution2}), we conclude that 
  if $0\notin\Omega$, there are no points satisfying 
  $$\phi(t^0_k, v_0, {\bf u}^0_k)\rightarrow v_1\hspace{.5cm}\mbox{ and }\hspace{.5cm}\phi(t^1_k, v_1, {\bf u}^1_k)\rightarrow v_0,$$
  that is, $\Sigma_{\R^2}$ admits no control set. 
\end{proof}

\begin{remark}
\label{eq}
    The proof of items (b) and (c) basically shows that the condition $0\notin\inner\Omega$ implies that the system $\Sigma_{\R^2}$ has no recurrent points other than the equilibria of the system.
\end{remark}

\subsubsection{The subcase $\tr A\neq 0$}

In this case, there exists an orthonormal basis $\{\mathbf{e}_1, \mathbf{e}_2\}$ of $\R^2$ where $A$ is written as $A=\left(\begin{array}{cc}
    \mu &  0\\
    0 & 0
\end{array}\right).$ The solutions of $\Sigma_{\R^2}$ for constant controls are given as
$$\varphi(t, v_0, u)=\left(\rme^{\mu t}\left(x_0+u\frac{\zeta_1}{\mu}\right)-u\frac{\zeta_1}{\mu}, y_0+u\zeta_2 t\right), \hspace{.5cm}\mbox{ where }\hspace{.5cm}v_0=(x_0, y_0).$$
For this choice of basis, the LARC is equivalent to $\zeta_1\zeta_2\neq 0$.

\begin{theorem}
    If the LCS $\Sigma_{\R^2}$ satisfies the LARC, $\det A=0$ and $\tr A=0$, it holds:
    \begin{itemize}
        \item[(a)] $0\in\inner\Omega$ and $\Sigma_{\R^2}$ admit a unique control set $\CC_{\R^2}$ which, in some basis, is written as
    $$\CC_{\R^2}=-\frac{\zeta_1}{\mu}\Omega\times\R\hspace{.5cm} \mbox{ if }\hspace{.5cm}\mu<0 \hspace{.5cm} \mbox{ and }\hspace{.5cm}\CC_{\R^2}=-\frac{\zeta_1}{\mu}\inner\Omega\times\R\hspace{.5cm} \mbox{ if }\hspace{.5cm}\mu>0.$$

\item[(b)] $0\in\partial\Omega$ and $\R\cdot \mathbf{e}_2$ is a continuum of one-point control sets.
    \item[(c)] $0\notin\Omega$ and $\Sigma_{\R^2}$ does not admit any control set.
    
    \end{itemize}

\end{theorem}

\begin{proof} (a) Let us consider the basis commented on at the beginning of the section and assume $\mu<0$. Since $\displaystyle{-\frac{\zeta_1}{\mu}\Omega\times\R}$ has a nonempty interior, by Proposition \ref{int}, we only have to show that it is maximal satisfying condition 2. in Definition \ref{control}. Assume that $u_1, u_2\in\Omega$ are such that 
    $$-\frac{\zeta_1}{\mu}\Omega\times\R=\left[-u_1\frac{\zeta_1}{\mu}, -u_2\frac{\zeta_1}{\mu}\right]\times\R.$$
    Then, for any $u\in \Omega$ and $t>0$, it holds that
    $$\varphi_1\left(t, \left(-u_1\frac{\zeta_1}{\mu}, y\right), u\right)+u_1\frac{\zeta_1}{\mu}=\rme^{-\mu t}\left(-u_1\frac{\zeta_1}{\mu}+u\frac{\zeta_1}{\mu}\right)-u\frac{\zeta_1}{\mu}+u_1\frac{\zeta_1}{\mu}=(\rme^{\mu t}-1)\frac{\zeta_1}{\mu}(-u_1+u)\geq 0,$$
    and 
    $$\varphi_1\left(t, \left(-u_1\frac{\zeta_1}{\mu}, y\right), u\right)+u_2\frac{\zeta_1}{\mu}=\rme^{-\mu t}\left(-u_2\frac{\zeta_1}{\mu}+u\frac{\zeta_1}{\mu}\right)-u\frac{\zeta_1}{\mu}+u_2\frac{\zeta_1}{\mu}=(\rme^{\mu t}-1)\frac{\zeta_1}{\mu}(-u_2+u)\leq 0,$$
    implying that 
    \begin{equation}
        \label{invariance}
\forall u\in\Omega, t>0, \hspace{1cm}\varphi\left(t, -\frac{\zeta_1}{\mu}\Omega\times\R, u\right)\subset -\frac{\zeta_1}{\mu}\Omega\times\R.
\end{equation}
    As a consequence, 
    \begin{equation}
        \label{1}
        \forall v\in -\frac{\zeta_1}{\mu}\Omega\times\R, \hspace{1cm}\OC^+(v)\subset -\frac{\zeta_1}{\mu}\Omega\times\R.
    \end{equation}
    
On the other hand, let $\displaystyle{v=\left(-u\frac{\zeta_1}{\mu}, y\right), v=\left(-u'\frac{\zeta_1}{\mu}, y'\right)\in-\frac{\zeta_1}{\mu}\inner\Omega}\times\R$  satisfying 
$$-u_1\frac{\zeta_1}{\mu}<-u\frac{\zeta_1}{\mu}<0<-u'\frac{\zeta_1}{\mu}<-u_2\frac{\zeta_1}{\mu}.$$
Since 
$$\varphi_1\left(t, \left(-u\frac{\zeta_1}{\mu}, y\right), u_2\right)\rightarrow -u_2\frac{\zeta_1}{\mu}, \hspace{.5cm}t\rightarrow+\infty,$$
there exists $t_1>0$ such that 
$$\varphi\left(t_1, \left(-u\frac{\zeta_1}{\mu}, y\right), u_2\right)=\left(-u'\frac{\zeta_1}{\mu}, y+\zeta_2u_2t_1\right).
$$
Since the time $t_1$ is independent of $y$, it holds that 
$$\varphi\left(s, \varphi\left(t_1, \varphi\left(t,\left(-u\frac{\zeta_1}{\mu}, y\right) ,u\right) ,u_2\right), u'\right)=\varphi\left(s, \varphi\left(t_1, \left(-u\frac{\zeta_1}{\mu}, y+u\zeta_2t\right) ,u_2\right), u'\right)$$
$$=\varphi\left(s, \left(-u'\frac{\zeta_1}{\mu}, y+u\zeta_2t+u_2\zeta_2t_1\right), u'\right)=\left(-u'\frac{\zeta_1}{\mu}, y+u\zeta_2t+u_2\zeta_2t_1+u'\zeta_2s\right).$$

Moreover, the fact that $uu'<0$ assures the existence of $s, t\geq 0$ such that $y+u\zeta_2t+u_2\zeta_2t_1+u'\zeta_2s=y'$, implying that one can go from $v$ to $v'$ in positive time. Analogously, interchanging the roles of $u$ and $u'$ gives us a trajectory in positive time from $v'$ to $v$. In particular,
\begin{equation}
    \label{2}
    \forall v\in -\frac{\zeta_1}{\mu}\inner\Omega\times\R, \hspace{.5cm }-\frac{\zeta_1}{\mu}\inner\Omega\times\R\subset\OC^+(v).
\end{equation}

By (\ref{1}) and (\ref{2}) we conclude that 
$$\forall v\in -\frac{\zeta_1}{\mu}\Omega\times\R, \hspace{.5cm} -\frac{\zeta_1}{\mu}\Omega\times\R=\overline{\OC^+(v)}.$$
Since the previous equality certainly implies maximality, the result is proved.

\bigskip

   As in the proof of Theorem \ref{polynomial}, if holds that 
$$
  \varphi_2(t, v, {\bf u})=y+\zeta_2\sum_{j=1}^nu_it_i,$$
  where $t_1, \ldots, t_n>0$ and
  $${\bf u}(s)=u_i\in\Omega\hspace{.5cm}\mbox{ for }\hspace{.5cm} s\in \left[\sum_{j=0}^{i-1}t_j, \sum_{j=0}^{i}t_j\right)\hspace{.5cm}\mbox{ with  }\hspace{.5cm}t_0=0\hspace{.5cm}\mbox{ and }\hspace{.5cm} t=\sum_{j=0}^{n}t_j.$$
  Therefore, $0\notin\inner\Omega$ implies that the only possible control sets are the equilibria of the system $\Sigma_{\R^2}$ (see Remark \ref{eq}). Then, $0\in\Omega$ and $\{\alpha\mathbf{e}_2\}, \alpha\in\R$ are the control sets of $\Sigma_{\R^2}$ showing (b), or $0\notin\Omega$ and $\Sigma_{\R^2}$ does not admit any equilibrium point, and hence, no control set.
\end{proof}

\subsection{The case $\det A\neq 0$.}

The present section analyzes the case where the determinant of $A$ is nonzero. As we will show, in this situation, the condition $0\in\Omega$ does not interfere in the system's behavior. The previous analysis will be divided on the possible signs of $\det A$.

\subsubsection{The case $\det A<0$.}

Since the eigenvalues of $A$ are real, the assumption $\det A<0$ forces that $A$ is a diagonalizable and
\begin{equation}
\label{formA1}
    A=\left(\begin{array}{cc}
\mu &0\\
0 & \lambda
\end{array}\right)\hspace{.5cm}\mu \lambda<0.
\end{equation}
in some orthonormal basis. Let us assume w.l.o.g. that in this situation it holds that $\lambda<0<\mu$.

The solutions of $\Sigma_{\R^2}$ for constant controls are given as
$$\phi(t, v_0, u)=\left(\rme^{\mu t}\left(x_0+u\frac{\zeta_1}{\mu}\right)-u\frac{\zeta_1}{\mu}, \rme^{\lambda t}\left(y_0+u\frac{\zeta_2}{\lambda}\right)-u\frac{\zeta_2}{\lambda}\right), \hspace{.5cm}\mbox{ where }\hspace{.5cm}v_0=(x_0, y_0),$$
and as in the previous case, the LARC is equivalent to $\zeta_1\zeta_2\neq 0$.

For any $u\in\Omega$, let us define the function 
    $$G_u:\R^2\rightarrow\R, \hspace{1cm}G_u(x, y):=\left|x+u\frac{\zeta_1}{\mu}\right|^{-\lambda}\left|y+u\frac{\zeta_2}{\lambda}\right|^{\mu},$$
    and note that 
    $$G_u(\phi(t, v_0, u))=\left|\phi_1(t, v_0, u)+u\frac{\zeta_1}{\mu}\right|^{-\lambda}\left|\phi_2(t, v_0, u)+u\frac{\zeta_2}{\lambda}\right|^{\mu}=\rme^{-\mu\lambda t}\left|x_0+u\frac{\zeta_1}{\mu}\right|^{-\lambda}\rme^{\lambda\mu t}\left|y_0-u\frac{\zeta_2}{\lambda}\right|^{\mu}=G_u(v_0).$$

   The next lemma will help us construct closed trajectories of $\Sigma_{\R^2}$.

\begin{lemma}
\label{controlinterior}
 Let $\alpha>0$ and $u_0, u_1\in\partial\Omega$ with $u_0\neq u_1$. For any $v\in \displaystyle{-\frac{\zeta_1}{\mu}\inner\Omega\times-\frac{\zeta_2}{\lambda}\inner\Omega}$ such that $G_{u_1}(v)\geq \alpha$, 
there exists $v_0, v_1\in \displaystyle{G_{u_1}^{-1}(\alpha)\cap G_{u_0}^{-1}\left(G_{u_0}(v)\right)}$ and times $t_0, t_1, t_2\geq 0$ such that 
$$\phi(t_0, v_0, u_0)=v, \hspace{.5cm}\phi(t_1, v, u_0)=v_1\hspace{.5cm}\mbox{ and }\hspace{.5cm}\phi(t_2, v_1, u_1)=v_0.$$
\end{lemma}

\begin{proof}
    Let us consider only the case $u_0=u^-$ and $u_1=u^+$, since the other possibility is analogous.

By the expression of the solutions, we get directly that
$$\phi_1(-t, v, u^-)\in-\frac{\zeta_1}{\lambda}\inner\Omega, 
\hspace{.5cm}\mbox{ and }\hspace{.5cm}\phi_2(t, v, u^-)\in-\frac{\zeta_2}{\lambda}\inner\Omega, \hspace{.5cm}\forall t>0.$$
Also, 
$$\phi_1(t, v, u^-)+u^+\frac{\zeta_1}{\mu}=\frac{\zeta_1}{\mu}\left(\underbrace{u^+-u^--\rme^{\mu t}(u-u^-)}_{=g(t)}\right),$$
with
$$g(0)=u^+-u^->0\hspace{.5cm}\mbox{ and }\hspace{.5cm}\lim_{t\rightarrow+\infty}g(t)=-\infty,$$
and hence, 
$$\phi(t'_1, v, u^-)=\left(-u^+\frac{\zeta_1}{\mu}, y\right), \hspace{.5cm}\mbox{ for some }\hspace{.5cm} t'_1>0\hspace{.5cm}\mbox{ and }\hspace{.5cm}y\in-\frac{\zeta_2}{\lambda}\inner\Omega,$$
Therefore, the continuous curve
$$t\in [0, t'_1]\mapsto G_{u^+}(\phi(t, v, u_0)),$$
is such that
$$G_{u^+}(v)=G_{u^+}(\phi(0, v, u^-))\hspace{.5cm}\mbox{ and }\hspace{.5cm}G_{u^+}(\phi(t'_1, v, u^-))=\left|-u^+\frac{\zeta_1}{\mu}+u^+\frac{\zeta_1}{\mu}\right|^{-\lambda}\left|y_1+u^+\frac{\zeta_2}{\lambda}\right|^{\mu}=0,$$
that together with the inequalities,
$$G_{u^+}(v)\geq\alpha>0=G_{u^+}(\phi(t'_1, v, u^-)),$$
imply the existence of $t_1\in [0, t'_1)$ such that $G_{u^+}(\phi(t_1, v, u^-))=\alpha$. 

Since 
$$G_{u^-}(\phi(t, v, u^-))=G_{u^-}(v), \hspace{.5cm} \forall t\in\R,$$
we get the point
$$v_1:=\phi(t_1, v, u^-)\in G^{-1}_{u^-}\left(G_{u^-}(v_0)\right)\cap G^{-1}_{u^+}\left(\alpha\right).$$

Analogously, 

$$\phi_2(-t, v, u^-)+u^+\frac{\zeta_2}{\lambda}=\frac{\zeta_2}{\lambda}\left(u^+-u^--\rme^{-\lambda t}(u-u^-)\right),$$
implies that
$$\phi(-t'_0, v, u^-)=\left(x, -u^+\frac{\zeta_2}{\lambda}\right), \hspace{.5cm}\mbox{ for some }\hspace{.5cm} t'_0>0\hspace{.5cm}\mbox{ and }\hspace{.5cm}x\in-\frac{\zeta_1}{\mu}\inner\Omega,$$
and hence, there exists $t_0\in [0, t'_0)$ such that 
$$v_0:=\phi(-t_0, v_0, u^-)\in G^{-1}_{u^-}\left(G_{u^-}(v_0)\right)\cap G^{-1}_{u^+}\left(\alpha\right).$$
By concatenation, we get
$$\phi(t_0, v_0, u^-)=\phi(t_0, \phi(-t_0, v, u^-), u^-)=v.$$

Certainly, $v_0\neq v_1$, otherwise, $v=v_0=v_1=-u^-A^{-1}\zeta$, which contradicts our initial assumption on $v$. Moreover, by construction, it holds that
$$v_0, v_1\in -\frac{\zeta_1}{\mu}\inner\Omega\times-\frac{\zeta_2}{\lambda}\inner\Omega.$$
Writing $v_i=\displaystyle{\left(-u'_i\frac{\zeta_1}{\mu}, -u''_i\frac{\zeta_2}{\lambda}\right)}$, we have by the previous that
$$v_1=\phi(t_0+t_1, v_0, u^-)\hspace{.5cm}\iff\hspace{.5cm}-u'_1\frac{\zeta_1}{\mu}=\phi_1(t_0+t_1, v_0, u^-)=\rme^{\mu(t_0+t_1)}\left(-u'_0\frac{\zeta_1}{\mu}+u^-\frac{\zeta_1}{\mu}\right)-u^-\frac{\zeta_1}{\mu}$$
$$\iff \hspace{.5cm}\rme^{\mu(t_0+t_1)}(-u'_0+u^-)=-u'_1+u^-\hspace{.5cm}\iff \hspace{.5cm}-u'_0+u^-<-u'_1+u^-\hspace{.5cm}\iff \hspace{.5cm}u'_0>u'_1.$$
Therefore, 
$$u^-<u'_1<u'_0<u^+\hspace{.5cm}\implies\hspace{.5cm}-u'_0+u^+>-u'_1+u^+,$$
and so, there exists $t_2>0$ such that 
$$\rme^{\mu t_2}(-u'_1+u^+)=-u'_0+u^+\hspace{.5cm}\iff \hspace{.5cm}\rme^{\mu t_2}\left(-u'_1\frac{\zeta_1}{\mu}+u^+\frac{\zeta_1}{\mu}\right)=-u'_0\frac{\zeta_1}{\mu}+u^+\frac{\zeta_1}{\mu},$$
$$\iff\hspace{.5cm}\phi_1(t_2, v_1, u^+)=\rme^{\mu t_2}\left(-u'_1\frac{\zeta_1}{\mu}+u^+\frac{\zeta_1}{\mu}\right)-u^+\frac{\zeta_1}{\mu}=-u'_0\frac{\zeta_1}{\mu}.$$

The previous, together with the equalities
$$G_{u^+}(\phi(t_2, v_1, u^+))=G_{u^+}(v_1),$$
imply that 
$$\phi_2(t_2, v_1, u^+)=-u'_0\frac{\zeta_2}{\lambda}\hspace{.5cm}\implies\hspace{.5cm}\phi(t_2, v_1, u^+)=-v_0,$$
concluding the proof.
\end{proof}
   
\begin{remark}
\label{remark}
Let us note that, in the hypothesis of Lemma \ref{controlinterior}, it holds that 
$$\{v_0, v_1\}=G_{u_1}^{-1}(\alpha)\cap G_{u_0}^{-1}(G_{u_0}(v)).$$
Moreover, $G_{u_1}(v)=\alpha$ implies necessarily that $v=v_0$ or $v=v_1$.
\end{remark}

  We are now in a position to prove the main result of this section.

    \begin{theorem}
    If the LCS $\Sigma_{\R^2}$ satisfies the LARC and $\det A< 0$, it admits a unique control set $\CC_{\R^2}$ satisfying

        $$\CC_{\R^2}=-\frac{\zeta_1}{\mu}\inner\Omega\times-\frac{\zeta_2}{\lambda}\Omega;$$

\end{theorem}

\begin{proof}  Let us start by constructing a closed orbit connecting two arbitrary points $v, w\in \displaystyle{-\frac{\zeta_1}{\mu}\inner\Omega\times-\frac{\zeta_2}{\lambda}\inner\Omega}$. Let us assume that $G_{u^+}(v)\leq G_{u^+}(w)$, otherwise, we change the roles of $v$ and $w$. By Lemma \ref{controlinterior}, there exists 
$$v_0, v_1\in G_{u^{+}}^{-1}(G_{u^+}(w))\cap G_{u^{-}}^{-1}(G_{u^-}(v)),$$
and times $t_0, t_1, t_2\geq 0$, such that 
$$\phi(t_0, v_0, u^-)=v, \hspace{.5cm}\phi(t_1, v, u^-)=v_1\hspace{.5cm}\mbox{ and }\hspace{.5cm}\phi(t_2, v_1, u^+)=v_0.$$

\begin{enumerate}
    \item If $G_{u^{-}}(w)\geq G_{u^-}(v)$ then necessarily $w\in \phi([0, t_2], v_1, u^+)$, implying that the closed trajectory given by the concatenation of the curves 
    $$t\in [0, t_1]\mapsto \phi(t, v_0, u^-)\hspace{.5cm}\mbox{ and }\hspace{.5cm}t\in [0, t_2]\mapsto \phi(t, v_1, u^+),$$
    is a solution of $\Sigma_{\R^2}$ passing through $v$ and $w$ (see Figure \ref{figura3});

    \item  If $G_{u^{-}}(w)< G_{u^-}(v)$, we can apply Lemma \ref{controlinterior} for $\alpha=G_{u^-}(w)$ and $w$ to obtain points $w_0, w_1$ and times $t'_0, t'_1, t_2'\geq 0$ satisfying
   $$\phi(t'_0, w_0, u^-)=w, \hspace{.5cm}\phi(t'_1, w, u^-)=w_1\hspace{.5cm}\mbox{ and }\hspace{.5cm}\phi(t'_2, w_1, u^+)=w_0.$$ 
   Since $G_{u^-}(v_i)\geq G_{u^-}(w_i)$, it holds that  $\phi(T_1, w_1, u^+)=v_1$ and $\phi(T_2, w_1, u^+)=v_0$ for some $T_1, T_2\in (0, t_2')$. Since by Remark \ref{remark} $w=w_0$ or $w=w_1$, the closed trajectory given by the concatenation of the curves 
   $$t\in [0, t_0+t_1]\mapsto \phi(t, v_0, u^-), \hspace{.5cm}t\in [T_1, t'_2]\mapsto \phi(t-T_1, v_1, u^+),\hspace{.5cm}t\in [0, t'_0+t_1']\mapsto \phi(t, w_0, u^-)$$
   $$\mbox{ and }\hspace{.5cm} t\in [0, t'_2-T_2]\mapsto \phi(t, w_1, u^+),$$
   is a solution of $\Sigma_{\R^2}$ passing through $v$ and $w$ (see Figure \ref{figura4}).
\end{enumerate} 

Therefore, $\displaystyle{ -\frac{\zeta_1}{\mu}\inner\Omega\times-\frac{\zeta_2}{\lambda}\Omega}$ has a nonempty interior and satisfies condition 2 in the Definition \ref{control}. Moreover, the fact that 
$$\varphi\left(-t, -\frac{\zeta_1}{\mu}\inner\Omega\times\R, u\right)\subset -\frac{\zeta_1}{\mu}\inner \Omega\times\R\hspace{.5cm}\mbox{ and }\hspace{.5cm}\varphi\left(t, \R\times -\frac{\zeta_2}{\lambda}\Omega, u\right)\subset \R\times -\frac{\zeta_2}{\lambda}\Omega, \hspace{.5cm}\forall t>0,$$
implies directly the maximality of $\displaystyle{ -\frac{\zeta_1}{\mu}\inner\Omega\times-\frac{\zeta_2}{\lambda}\Omega}$, which by Proposition \ref{int} implies that it is a control set and concludes the proof.
\end{proof}

\begin{figure}[h]
	\centering
	\begin{subfigure}{.5\textwidth}
		\centering
		\includegraphics[width=.75\linewidth]{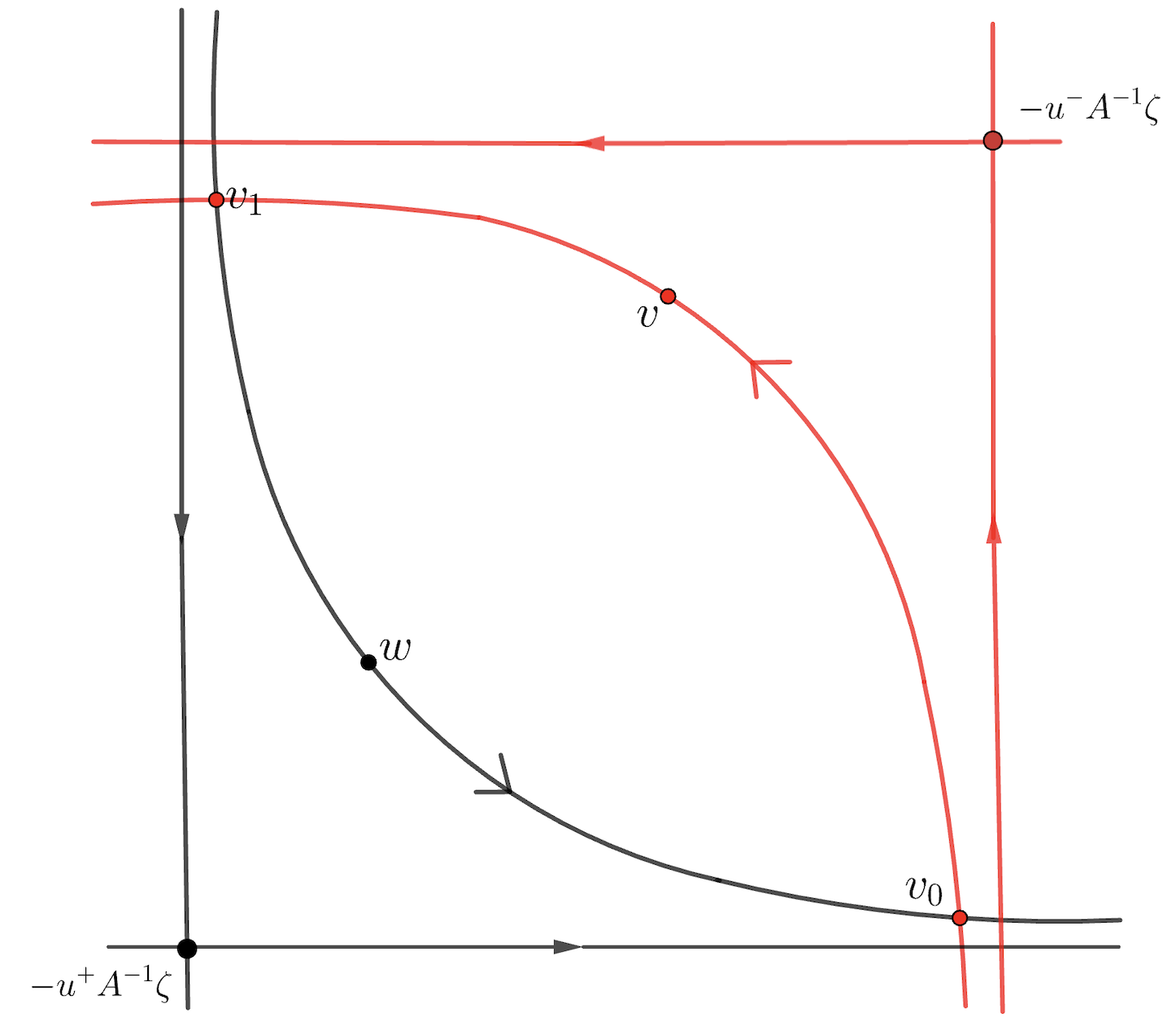}
		\caption{The case $G_{u^-}(w)\geq G_{u^-}(v)$.}
		\label{figura3}
	\end{subfigure}%
	\begin{subfigure}{.5\textwidth}
		\centering
		\includegraphics[width=.8\linewidth]{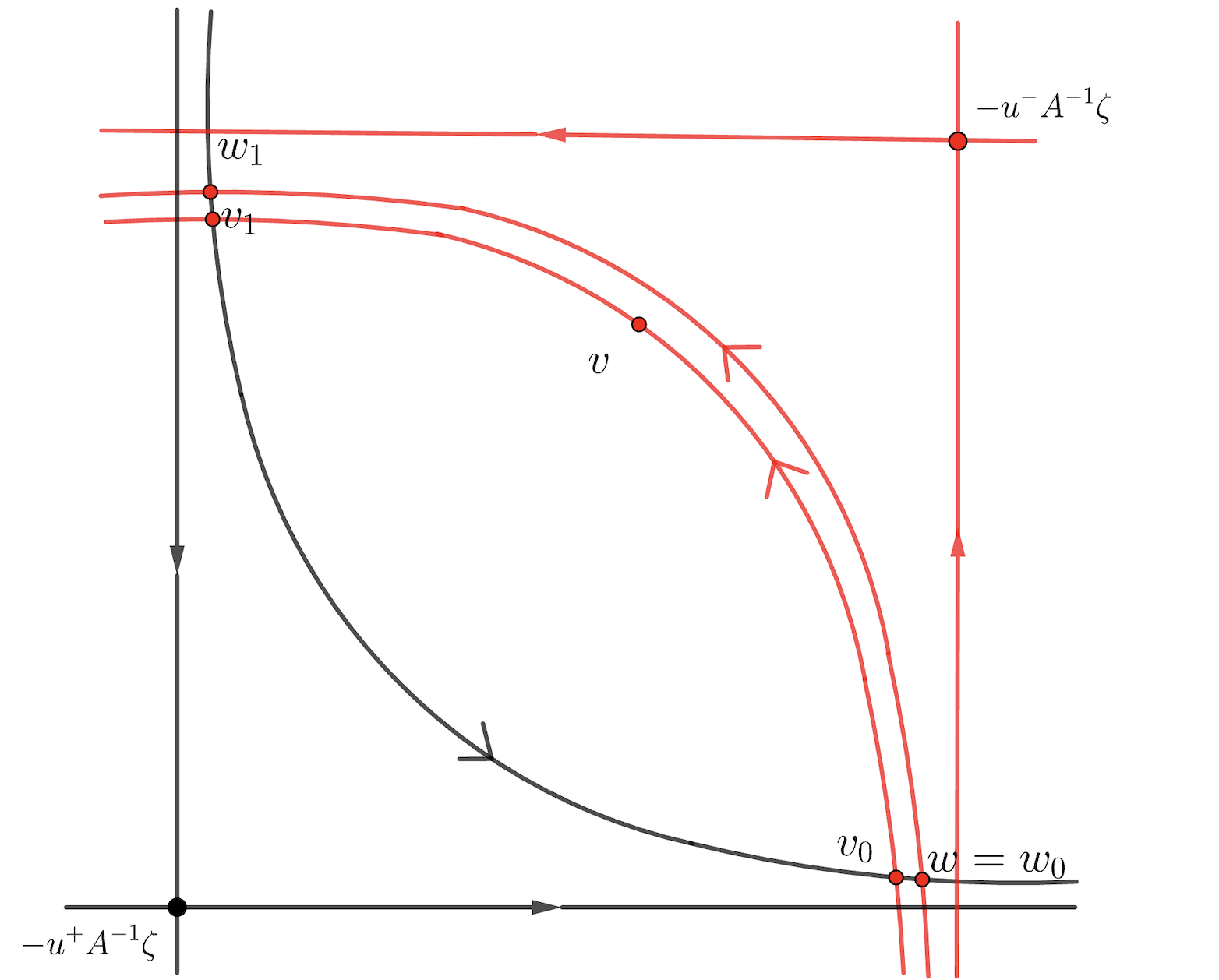}
		\caption{The case $G_{u^-}(w)< G_{u^-}(v)$ and $w=w_0$.}
		\label{figura4}
	\end{subfigure}
	\caption{Closed trajectory through $v$ and $w$.}
\end{figure}

\subsubsection{The case $\det A>0$.}

In this section, we will assume that the real eigenvalues of $A$ have the same sign. In this case, there exists a basis of $\R^2$ such that

\begin{equation}
\label{formsA}
A=\left(\begin{array}{cc}
\lambda&1\\
0 & \lambda
\end{array}\right)\hspace{.5cm}\mbox{ or }\hspace{.5cm}A=\left(\begin{array}{cc}
\lambda&0\\
0 & \mu
\end{array}\right)\hspace{.5cm}\lambda\mu>0.
\end{equation}

Since we are in the case $\det A>0$, let us assume, w.l.o.g., that the eigenvalues of $A$ are negatives, since the positive case is analogous.

Let $u_1, u_2\in\Omega$ with $u_1<u_2$ and define the function 
\begin{equation}
    \label{functionf}
f:(0,+\infty)^2\rightarrow \R^2, \hspace{1cm}
f(s,t)=\varphi(s,\varphi(t,v_1,u_2),u_1),
\end{equation}
where for simplicity we put $v_i=v(u_i)=-u_iA^{-1}\zeta$. Since $\det A\neq 0$, we have that 
$\varphi(t,v,u)=\rme^{tA}(v-v(u))+v(u)$, 
 and hence, $f$ can be rewritten as 
$$f(s,t)=(u_2-u_1)A^{-1}\rme^{sA}(\rme^{tA}-I)\zeta- u_1A^{-1}\zeta.$$

On the other hand, 
\begin{equation*}
\frac{\partial f}{\partial s}=(u_2-u_1)\rme^{sA}(\rme^{tA}-I)\zeta\ \ \ \ \ \textrm{e}\ \ \ \ \ \frac{\partial f}{\partial t}=(u_2-u_1)\rme^{(t+s)A}\zeta,
\end{equation*}
and, if for some $a\in \R$, it holds that 

$$\frac{\partial f}{\partial s}=a\frac{\partial f}{\partial t}\hspace{.5cm}\iff\hspace{.5cm}\rme^{(s+t)A}\zeta-\rme^{sA}\zeta = a\rme^{(t+s)A}\zeta\hspace{.5cm}\iff\hspace{.5cm}(1- a)\rme^{(t+s)A}\zeta =\rme^{sA}\zeta$$
$$\iff\hspace{.5cm}(1- a)\rme^{tA}\zeta= \zeta\hspace{.5cm}\stackrel{\zeta\neq 0}{\iff}\hspace{.5cm}\zeta\mbox{ is eigenvalue of } \rme^{tA}.$$
However, the possibilities for $A$ given in (\ref{formsA}) implies necessarily that 
 any eigenvalue of $\rme^{tA}$ is also an eigenvalue of $A$ and hence
 $$\left\{\frac{\partial f}{\partial s}, \frac{\partial f}{\partial t}\right\}\mbox{ is linearly independenty }\hspace{.5cm}\iff \hspace{.5cm}\langle A\zeta, \theta\zeta\rangle\neq 0.$$

The next lemma implies that the function $f$ defined previously is actually injective.

\begin{lemma}
\label{nocross}
Let $\tau_0>0$ and write $v=\varphi(\tau_0, v_1, u_2)$. If $r$ is the line through $v_1$ and $v$, then the curves  
\begin{equation*}
s\in(0,\tau_0)\longrightarrow \varphi(s, v_1, u_2)\hspace{.5cm}\mbox{ and }\hspace{.5cm} t\in(0,+\infty)\longrightarrow \varphi(t, v, u_1),
\end{equation*}
are on different half planes determined by $r$ (Figure \ref{bothsides} below).
\end{lemma}
\begin{figure}[H]
\centering
\includegraphics[scale=0.6]{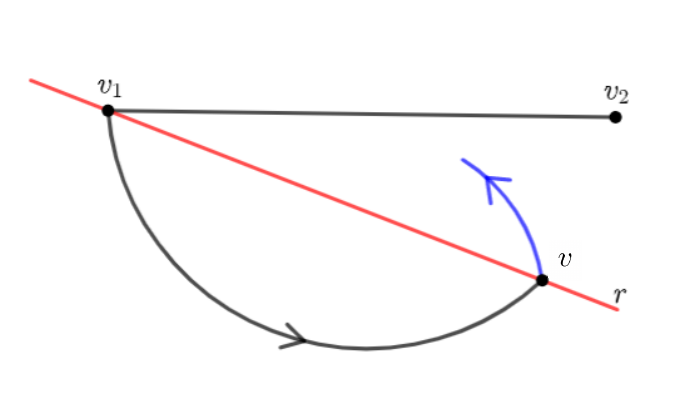}
\caption{Curves on different half spaces.}\label{bothsides}
\end{figure}

\begin{proof} In order to show the result, it is enough to show that the functions
\begin{equation*}
\rho_1(t)=\left\langle \varphi(t, v_1, u_2)-v_1, \theta(v-v_1) \right\rangle \hspace{.5cm} \mbox{ and }\hspace{.5cm} \rho_2(s)=\left\langle \varphi(s, v,u_1)-v_1,\theta(v-v_1) \right\rangle
\end{equation*}
have opposite signs. Since, 
\begin{equation}\label{autovetordeA}
\begin{array}{rcl}
v-v_1&=&\varphi(\tau_0, v_1,u_2)-v_1\\
&=&\left(\rme^{\tau_0A}-I\right)(v_1-v_2)\\
&=&(u_2-u_1)A^{-1}\left(\rme^{\tau_0A}-I\right)\zeta,
\end{array}
\end{equation}
the vector $v-v_1$ is an eigenvector of $A$ if and only if $\langle A\zeta,\theta\zeta\rangle=0$. Therefore, the function $\rho_1$ is trivially zero. On the other hand, a simple derivation gives us that
$$
\rho_1'(t)=\left\langle A\rme^{tA}(v_1-v_2),\theta(v-v_1) \right\rangle= (u_2-u_1)^2\left\langle \rme^{tA}\zeta,\theta A^{-1}\left(\rme^{\tau_0A}-I\right)\zeta \right\rangle$$
$$=\frac{(u_2-u_1)^2}{\det A}\det\left(\rme^{\tau_0A}-I\right) \left\langle A\rme^{tA}\left(\rme^{\tau_0A}-I\right)^{-1}\zeta,\theta \zeta\right\rangle,
$$
and hence, $\rho_1'(t)=0$ if and only if $\zeta$ is an eigenvector of $B:=A\rme^{t_0A}\left(\rme^{\tau_0A}-I\right)^{-1}$. Since $A$ and $B$ commute and  $\langle A\zeta,\theta\zeta\rangle\neq0$, we have necessarily that $B=cI$. Using the possible forms of $A$ given in (\ref{formsA}), we conclude that 
\begin{equation*}
c=\frac{\lambda\rme^{t\lambda}}{\rme^{\tau_0\lambda}-1}>0, \hspace{.5cm}\mbox{ by the assumption }\hspace{.5cm}\lambda<0,
\end{equation*}
and $\rho'_1(t)=0$ for, at most, one $t\in (0, \tau_0)$. By Rolle's Theorem and the fact that $\rho_1(0)=\rho_1(\tau_0)=0$, there exists a unique $t_0\in(0,\tau_0)$ such that $\rho_1'(t_0)=0$. Furthermore, $\rho_1$ cannot change signs on $(0, \tau_0)$. 




A second derivation gives us  

\begin{equation*}
\rho_1''(t)=\frac{(u_2-u_1)^2}{\det A}\det(e^{\tau_0A}-1)\left\langle A^2\rme^{tA}\left(\rme^{\tau_0A}-I\right)^{-1}\zeta,\theta \zeta \right\rangle,
\end{equation*}
implying that
\begin{equation*}
\rho_1''(t_0)=c\frac{(u_2-u_1)^2}{\det A}\det(e^{\tau_0A}-1)\left\langle A\zeta,\theta \zeta \right\rangle,
\end{equation*}
implying that the sign of $\rho_1$ is the same as the value of $-\langle A\zeta,\theta \zeta \rangle$.

On the other hand, let us notice that
$$0=\rho_2(s)=\left\langle  \varphi(s, v ,u_1)-v_1 ,\theta(
v -v_1)\right\rangle=\left\langle  \rme^{sA}(v-v_1),\theta(v-v_1)\right\rangle,$$
if and only if $v-v_1$ is an eigenvector of $\rme^{sA}$ if and only if $v-v_1$ is an eigenvector of $A$. Since the previous does not happen, we conclude that $\rho_2\neq 0$ for all $t>0$. However, the fact that
$$\lim_{s\rightarrow 0^+}\rho_2(s)=\lim_{s\rightarrow \infty^+}\rho_2(s)=0,$$
implies that $\rho_2'(s)=0$ for some $s>0$. By a simple derivation, we get 
$$\rho_2'(s_0)=\left\langle  A\rme^{s_0A}(v-v_1),\theta(v-v_1)\right\rangle=0,$$
which is equivalent to if $v-v_1$ be an eigenvector of $B':=A\rme^{s_0A}$. As previously stated, the commutativity of $A$ and $B'$ together with the fact that $v-v_1$ is not an eigenvector of $A$, forces the existence of $c'\in\R$ such that $B'=c'I$. Using (\ref{formsA}) we conclude that 
$$c'=\lambda\rme^{s_0\lambda}<0, \hspace{.5cm}\mbox{ by the assumption }\hspace{.5cm}\lambda<0,$$
and that $s_0$ is unique. Consequently, the sign of $\rho_2$ is the same as that of $\rho_2( s_0)$ and 
$$ \rho_2(s_0)=\left\langle  \rme^{s_0A}(v-v_1),\theta(v-v_1)\right\rangle=
c'\left\langle  A^{-1}(v-v_1),\theta(v-v_1)\right\rangle = c'\frac{(u_2-u_1)^2}{\det(A)}\det(\rme^{\tau_0A}-1)\left\langle  A^{-1}\zeta,\theta\zeta \right\rangle$$
$$=c'\frac{(u_2-u_1)^2}{\det A}\det(\rme^{\tau_0A}-1)\left\langle  \zeta,\theta A\zeta \right\rangle=-c'\frac{(u_2-u_1)^2}{\det A}\det(\rme^{\tau_0A}-1)\left\langle  A\zeta,\theta\zeta \right\rangle.
$$
showing that the sign of $\rho_2$ is equal to the sign of $\left\langle  A\zeta,\theta\zeta \right\rangle$ and concluding the proof.
\end{proof}

\bigskip

The previous lemma implies that the function $f$ defined in (\ref{functionf}) is injective. In fact, 
\begin{equation*}
f(s_1,t_1)=f(s_2,t_2)\hspace{.5cm}\iff\hspace{.5cm}\rme^{s_1A}\left(\rme^{t_1A}-I \right)\zeta=\rme^{s_2A}\left(\rme^{t_2A}-I \right)\zeta,
\end{equation*}
which shows that $\zeta$ is an eigenvector of $B'':=\left(\rme^{s_2A}\left(\rme^{t_2A}-I \right)\right)^{-1}\rme^{s_1A}\left(\rme^{t_1A}-I \right)$. Again, the fact that $A$ and $B''$ commute, together with $\left\langle  A\zeta,\theta\zeta \right\rangle\neq 0$, implies that $B''=I$ and hence 
\begin{equation*}
\rme^{s_1A}\left(\rme^{t_1A}-I \right)=\rme^{s_2A}\left(\rme^{t_2A}-I \right).
\end{equation*}
Since
$$s_1<s_2\hspace{.5cm}\iff\hspace{.5cm}t_1<t_2\hspace{.5cm}\mbox{ and }\hspace{.5cm} s_1=s_2\hspace{.5cm}\iff\hspace{.5cm}t_1=t_2,$$ 
we can assume, w.l.o.g., that $s_1<s_2$ and $t_1<t_2$. Denoting by $v'=\varphi(t_1, v_1, u_2)$ and $v=\varphi(t_2, v_1, u_2)$, we get that 
\begin{equation*}
\varphi(s_1,v',u_1)=f(s_1,t_1)=f(s_2,t_2)=
\varphi(s_2,v,u_1), 
\end{equation*}
or equivalently, $v'=\varphi(s_2-s_1,v,u_1)$. However, such equality contradicts Lemma \ref{nocross}
, since $v'$ and $\varphi(s_2-s_1, v,u_1)$ have to be on different half planes of the line through $v_1$ and $v$. Therefore, 
$$(s_1,t_1)\neq(s_2,t_2)\hspace{.5cm}\implies\hspace{.5cm}f(s_1,t_1)\neq f(s_2,t_2),$$
showing that $f$ is injective. A geometrical description of the function $f$ is given by Figure \ref{graphf} below.

\begin{figure}[H]
\centering
\includegraphics[scale=0.5]{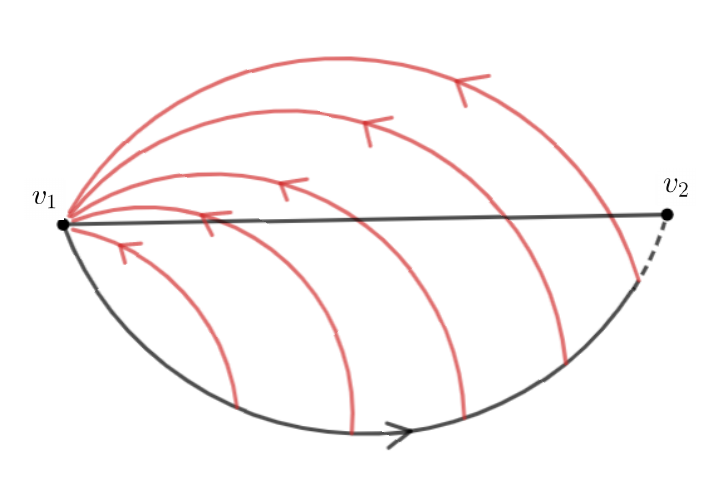}
\caption{Geometrical description of $f$.}\label{graphf}
\end{figure}

Since we are assuming that $A$ has only negative eigenvalues, by \cite[Proposition 2.2.7]{FCWK1} there exists a norm $|\cdot|_A$ of $\R^2$ and a positive number $\delta>0$ such that 
\begin{equation}
\label{desigualdade}
\left|\rme^{tA}v\right|_A\leq \rme^{-t\delta }|v|_A, \hspace{.5cm}\forall t>0, v\in\R^2.
\end{equation}
This fact will help us to prove the uniqueness of the control set of $\Sigma_{\R^2}$, as stated in the next result.

\begin{theorem}
\label{theo}
    If the LCS $\Sigma_{\R^2}$ satisfies the LARC and $\det A>0$, it admits a unique control set with a nonempty interior $\CC_{\R^2}$ satisfying
     $$\inner\CC_{\R^2}=\mathrm{Im}(f),$$
     where $f$ is the diffeomorphism
    $$ f:(0,+\infty)^2\rightarrow \R^2, \hspace{1cm}
f(s,t)=\varphi(\epsilon s,\varphi(\epsilon t, v(u^-), u^+), u^-),\hspace{.5cm}\mbox{ with }\hspace{.5cm}\epsilon=\left\{\begin{array}{cc}
    1 & \mbox{ if }\tr A<0 \\
    -1 & \mbox{ if }\tr A>0
\end{array}\right..$$

Moreover, $\CC_{\R^2}$ is closed if $\tr A<0$ and open if $\tr A>0$. 
\end{theorem}

\begin{proof}
     Let us assume, as previously stated, that $A$ has only negative eigenvalues, or equivalently $\tr A<0$. The fact that $f$ is a diffeomorphism was stated previously. By the definition of $f$, it holds that $\mathrm{Im}(f)\subset\OC^+(v(u^-))$. On the other hand, 
    $$\varphi(t, v, u^-)\rightarrow v(u^-)\hspace{.5cm}\implies\hspace{.5cm}\overline{\OC^+(v(u^-))}\subset \overline{\OC^+(v)},$$
    and hence
    $$\overline{\mathrm{Im}(f)}\subset\overline{\OC^+(v(u^-))}\subset \overline{\OC^+(v)}, \hspace{.5cm}\forall v\in \overline{\mathrm{Im}(f)}.$$
    Therefore, $\overline{\mathrm{Im}(f)}$ is a subset with a nonempty interior satisfying condition 2. in Definition \ref{control}. In order to show that it is in fact a control set, let us show that $\mathrm{Im}(f)$ is positively invariant. Let us note that the property
    \begin{equation}
        \label{solution3}
        \varphi(t, v, u)\rightarrow v(u), \hspace{.5cm}t\rightarrow+\infty,
    \end{equation}
    implies that if a trajectory of $\Sigma_{\R^2}$ leaves $\mathrm{Im}(f)$, it has to come back inside $\mathrm{Im}(f)$, crossing the boundary $\partial(\mathrm{Im}(f))$ (at most) once. Therefore, the angle at intersection points is zero if the trajectory crosses $\partial(\mathrm{Im}(f))$ once, and it has different signs if the trajectory crosses $\partial(\mathrm{Im}(f))$ twice. 
    
    Since the boundary $\partial(\mathrm{Im}(f))$ is parametrized by the curves, 
    $$t\in(0, +\infty)\mapsto\varphi(t, v(u^-), u^+)\hspace{.5cm}\mbox{ and }\hspace{.5cm}t\in(0, +\infty)\mapsto\varphi(t, v(u^+), u^-),$$
    in order to show that no trajectory leaves $\mathrm{Im}(f)$ in positive time,  
    it is enough to show that 
    $$\left\langle Av+u\zeta, \theta\left(Av+u^*\zeta\right)\right\rangle,\hspace{.5cm}\forall u\in\Omega, \hspace{.5cm} u\neq u^*,$$
    where $v\in \varphi((0, +\infty), v(u^-), u^+)$ if $u^*=u^+$ and $v\in \varphi((0, +\infty), v(u^+), u^-)$ if $u^*=u^-$, have all the same sign (see Figure \ref{sign}).

    However, the fact that.
    $$A\varphi(t, v(u^+), u^-)=A\left(\rme^{tA}(v(u^+)-v(u^-))+v(u^-)\right)=(u^--u^+)\rme^{tA}\zeta-u^-\zeta$$
 and 
 $$A\varphi(t, v(u^-), u^{+})=A\left(\rme^{tA}(v(u^-)-v(u^+))+v(u^+)\right)=(u^+-u^-)\rme^{tA}\zeta-u^+\zeta,$$
implies necessarily
    $$\left\langle Av+u\zeta, \theta\left(Av+u^*\zeta\right)\right\rangle=|u^*-u|(u^+-u^-)\langle \rme^{tA}, \theta\zeta\rangle,$$
    and since, by the LARC, $\langle \rme^{tA}\zeta, \theta\zeta\rangle\neq 0$ for all $t\neq 0$, we conclude that $\mathrm{Im}(f)$ is positively invariant, implying that $\CC_{\R^2}=\overline{\mathrm{Im}(f)}$ a control set of $\Sigma_{\R^2}$, with $\inner\CC_{\R^2}=\mathrm{Im}(f)$.

For the uniqueness, let us notice that, since $\CC_{\R^2}$ is compact, for any $v\in\R^2\setminus\CC_{\R^2}$ there exists $v_0\in\CC_{\R^2}$ such that 
$$\left|v-\CC_{\R^2}\right|_A=\inf\left\{\left|v-w\right|_A, w\in\CC_{\R^2}\right\}=\left|v-v_0\right|_A.$$
Hence, the fact that $\OC^+(v_0)\subset\CC_{\R^2}$ implies
$$\left|\varphi(t, v, u)-\CC_{\R^2}\right|_A\leq \left|\varphi(t, v, u)-\varphi(t, v_0, u)\right|_A= \left|\rme^{tA}(v-v_0)\right|_A\leq \rme^{\delta t}|v-v_0|_A=\rme^{-\delta t}|v-\CC_{\R^2}|_A,$$
where for the last inequality we used relation (\ref{desigualdade}). Now, if $\CC'$ is a second control set of $\Sigma_{\R^2}$ and we consider $v'\in \CC'$, condition 1. in Definition \ref{control}, implies the existence of ${\bf u}\in\UC$ such that $\varphi(t, v', {\bf u})\in \CC'$ for all $t>0$. Hence, 
$$\inf\{|v-w|_A, v\in\CC_{\R^2}, w\in\CC'\}\leq |\varphi(t, v', {\bf u})-\CC_{\R^2}|_A\leq\rme^{-\delta t}|v'-\CC_{\R^2}|_A\rightarrow 0,$$
which by the compactness of $\CC_{\R^2}$ implies that $\CC'\subset\partial \CC_{\R^2}$ and hence $\CC'=\CC_{\R^2}$, showing the uniqueness of $\CC_{\R^2}$ and concluding the proof. 
    \end{proof}

\begin{figure}[H]
\centering
\includegraphics[scale=0.5]{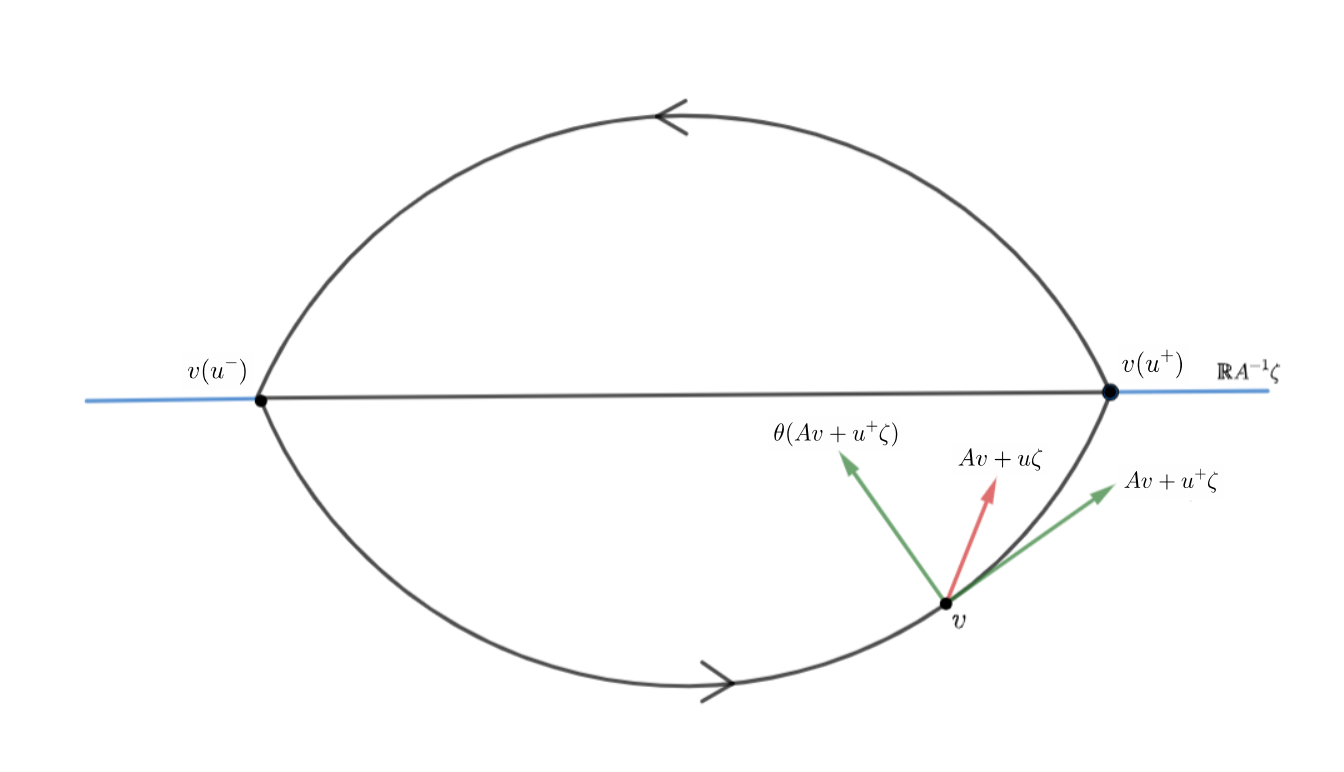}
\caption{Positively invariance of $\mathrm{Im}(f)$.}\label{sign}
\end{figure}

We finish this section with an interesting result concerning the open control set.

\begin{corollary}
    If the LCS $\Sigma_{\R^2}$ satisfies the LARC, $\det A>0$ and $\tr A>0$, the open control set $\CC_{\R^2}$ admits two one-point control sets $\{v(u^+)\}$ and $\{v(u^-)\}$ at its boundary $\partial\CC_{\R^2}$.
\end{corollary}

\begin{proof}
By reversing the time in the proof of Theorem \ref{theo} if $\tr A>0$, it holds that $\CC_{\R^2}=\mathrm{Im}(f)$ is an open control set of $\Sigma_{\R^2}$ and any other possible control set has to be contained in $\partial\CC_{\R^2}$. Moreover, $\CC_{\R^2}$ is negatively invariant, implying that
$$\left|\varphi(t, v, u)-\CC_{\R^2}\right|_A\geq\rme^{\delta t}|v-\CC_{\R^2}|_A, \hspace{.5cm}\forall t>0, v\in\R^2\setminus\overline{\CC_{\R^2}}.$$

Therefore, if $v_1, v_2\in\partial\CC_{\R^2}$ belongs to a control set $\CC'\subset \partial\CC_{\R^2}$ and 
$$\varphi(t_k, v_1, {\bf u}_k)\rightarrow v_2, \hspace{.5cm}k\rightarrow+\infty,\hspace{.5cm}\mbox{ then }\hspace{.5cm}\varphi(t_k, v_1, {\bf u}_k)\in\partial\CC_{\R^2}, \hspace{.5cm}\forall k\in\N.$$

Since, $\partial\CC_{\R^2}$ is parametrized by the curves, 
    $$t\in(0, +\infty)\mapsto\varphi(-t, v(u^-), u^+)\hspace{.5cm}\mbox{ and }\hspace{.5cm}t\in(0, +\infty)\mapsto\varphi(-t, v(u^+), u^-),$$
we must have that ${\bf u}_k\equiv u^+$ or ${\bf u}_k\equiv u^-$, which implies $\{v(u^+)\}$ and $\{v(u^-)\}$ are one-point control sets, concluding the proof.
\end{proof}

\begin{remark}
It is worth to emphasize the differences between the real case, treated here, and the complex case, treated in \cite{VAAS}. Firstly, the fact that $\Omega$ is or is not an interval containing zero has no difference in the complex case, though it strongly influences the dynamics in the real case. Secondly, in the complex case, if $\det A>0$ and $\tr A>0$, the whole boundary of $\CC_{\R^2}$ is a control set with an empty interior, contrasting with the two distinct one-point control sets in the real case.

One exciting behavior for both cases, is the existence of a control set with a nonempty interior when $\det A\neq0$, independently of the control range $\Omega$.
\end{remark}

\end{document}